\documentclass[10pt,reqno]{amsart}
\title[Higher regularity for parabolic equations with singular drift]{Higher regularity of H\"older continuous solutions of parabolic equations with singular drift velocities}
\author{Susan Friedlander}
\address{Department of Mathematics,
University of Southern California, 3620 S.~Vermont Ave.,
Los Angeles, CA 90089} \email{\tt susanfri@usc.edu}

\author{Vlad Vicol}
\address{Department of Mathematics, University of Chicago, 5734 University Ave., Chicago, IL 60637}
\email{\tt vicol@math.uchicago.edu}
\usepackage{amsmath, amsthm, amssymb}
\usepackage{color}
\usepackage{times}

\theoremstyle{plain}
\newtheorem{theorem}{Theorem}

\newtheorem{lemma}[theorem]{Lemma}
\newtheorem{proposition}[theorem]{Proposition}
\newtheorem{corollary}[theorem]{Corollary}

\theoremstyle{definition}

\def\tilde{\widetilde}
\numberwithin{equation}{section}

\newcommand\BB[2]{\dot{B}_{#2,\infty}^{#1}}

\renewcommand\hat{\widehat}
\def\ZZ3{{\mathbb Z}^3}
\def\RR1{{\mathbb R}}
\def\RR2{{\mathbb R}^2}
\def\RR3{{\mathbb R}^3}
\def\RRd{{\mathbb R}^d}
\def\TT2{{\mathbb T}^2}

\def\eps{\varepsilon}

\def\II{\mathcal I}

\def\Dj{\Delta_j}

\begin{document}

%%%%%%%%%%%%%%%%%%%%%%%%% THE ABSTRACT %%%%%%%%%%%%%%%%%%%%%%%%%%%%%%%%%%%
\begin{abstract}
Motivated by an equation arising in magnetohydrodynamics, we prove that H\"older continuous weak solutions of a nonlinear parabolic equation with singular drift velocity are classical solutions. The result is proved using the space-time Besov spaces introduced by Chemin and Lerner, combined with energy estimates, without any minimality assumption on the H\"older exponent of the weak solutions.
\end{abstract}

\subjclass[2000]{76D03, 35Q35}

\keywords{Magneto-geostrophic model, higher regularity, weak solutions, space-time Besov spaces.}

\maketitle
%%%%%%%%%%%%%%%%%%%%%%%%%% The Main Part %%%%%%%%%%%%%%%%%%%%%%%%%%%%%%%%%%
\section{Introduction}\label{sec:intro}\setcounter{equation}{0}

In this paper we address the smoothness of H\"older continuous weak solutions of the scalar nonlinear parabolic equation with singular drift velocity
\begin{align}
  & \partial_t \theta - \Delta \theta + (u\cdot\nabla) \theta = 0 \label{eq:N1}\\
  & u_j = \partial_i T_{ij} \theta\label{eq:N2}
\end{align}on $\RRd \times (0,\infty)$, where $d\geq 2$, $\{T_{ij}\}_{i,j=1}^{d}$ is a $d\times d$ matrix of Calder\'on-Zygmund singular integral operators, and the summation convention on repeated indices is used throughout. The drift velocity is taken to be divergence-free, i.e.
\begin{align}
    & \nabla \cdot u = 0 \label{eq:N3}
\end{align}which is ensured by \eqref{eq:N2} if the matrix $\{T_{ij}\}$ is taken such that $\partial_i \partial_j T_{ij} f = 0$ for smooth functions $f$. The system \eqref{eq:N1}--\eqref{eq:N3} is supplemented with the initial condition
\begin{align}
  & \theta(\cdot,0) = \theta_0\label{eq:N4}
\end{align}where $\theta_0 \in L^2(\RRd)$ has zero mean on $\RRd$. We note that $\int_{\RRd}\theta(x,t)\, dx$ is conserved in time by solutions of \eqref{eq:N1}--\eqref{eq:N2}.

The global in time existence of finite energy weak solutions to \eqref{eq:N1}--\eqref{eq:N4} has been proven by the authors of this paper in~\cite{FriedlanderVicol}. Additionally, in \cite{FriedlanderVicol} we prove that for positive time the weak solutions are in fact H\"older continuous (see also \cite{SSSZ}). In the present paper we address the higher regularity of these H\"older continuous weak solutions, by proving that they are classical solutions (even $C^\infty$ smooth) for positive time. This result was announced in \cite[Lemma 3.4]{FriedlanderVicol}.

The motivation for studying advection-diffusion equations with drift velocities as singular as those given by \eqref{eq:N2} came from the three-dimensional equations of magneto-geostrophic dynamics, for a rapidly rotating, electrically conducting fluid (cf.~Moffatt~\cite{Moffatt}). A well studied advection-diffusion equation also arising in rotating fluids (cf.~Constantin, Majda, and Tabak~\cite{ConstantinMajdaTabak}) is the so-called critical surface quasi-geostrophic (SQG) equation. This 2-dimensional equation has the form
\begin{align}
  &\partial_t \theta + (-\Delta)^{1/2} \theta + (u \cdot \nabla) \theta = 0\label{eq:SQG:1}\\
  &u = \nabla^\perp (-\Delta)^{-1/2} \theta\label{eq:SQG:2}.
\end{align}Although there are significant differences between the systems \eqref{eq:N1}--\eqref{eq:N4} and \eqref{eq:SQG:1}--\eqref{eq:SQG:2}, in both cases $L^\infty(\RRd)$ is the critical Lebesgue space with respect to the natural scaling of the equations. The criticality of the $L^\infty$ norm with respect to scaling also holds for the modified critical surface quasi-geostrophic equation considered by Constantin, Iyer, and Wu in~\cite{ConstIyerWu}
\begin{align}
  &\partial_t \theta + (-\Delta)^{\beta/2} \theta + (u \cdot \nabla ) \theta = 0 \label{eq:MQG:1}\\
  & u= \nabla^\perp (-\Delta)^{\beta/2-1} \theta \label{eq:MQG:2}
\end{align}where $\beta \in (0,1)$. In a recent paper, Caffarelli and Vasseur~\cite{CaffVass} used De Giorgi iteration to prove that weak solutions of \eqref{eq:SQG:1}--\eqref{eq:SQG:2}, with $L^2$ initial data, are smooth for positive time. A different proof of global regularity for \eqref{eq:SQG:1}--\eqref{eq:SQG:2} was given independently by Kiselev, Nazarov, and Volberg~\cite{KiselevNazarovVolberg} (see also Kiselev and Nazarov~\cite{KiselevNazarov}). The proof of H\"older regularity of weak solutions to \eqref{eq:N1}--\eqref{eq:N4} given by the authors of the present paper in~\cite{FriedlanderVicol} is also based on a suitable modification of the De Giorgi method, along the lines of \cite{CaffVass}. Once the weak solutions to \eqref{eq:N1}--\eqref{eq:N4} are H\"older continuous, we expect to be able to bootstrap to higher regularity, since the H\"older $C^\alpha$ norm is subcritical with respect to the natural scaling of the equations, for \textit{any} $\alpha>0$. This matter is however not automatic due to the singular velocity drift $u$, which by \eqref{eq:N2} lies in $L^\infty_t C^{\alpha-1}_x \cap L^2_{t,x}$, whenever $\theta \in L^\infty_t C_x^\alpha \cap L^2_t \dot{H}^{1}_{x}$. The following theorem is the main result of the present paper.

\begin{theorem}\label{thm:higher}
Let $\theta_0\in L^2(\RRd)$ be given. Let
\begin{align}
  \theta \in L^\infty([0,\infty);L^2(\RRd)) \cap L^2((0,\infty);\dot{H}^1(\RRd)) \cap L^\infty((0,\infty);C^\alpha(\RRd)) \label{eq:thm:assumption}
\end{align}be a H\"older continuous weak solution of \eqref{eq:N1}--\eqref{eq:N4}, evolving from $\theta_0$, where $\alpha\in(0,1)$ is given. Then, the solution is classical, that is
\begin{align}
\theta \in L^\infty([t_0,\infty); C^{1,\delta}(\RRd))\label{eq:thm:conclusion}
\end{align}
for any $t_0>0$ and $\delta\in(0,1)$.
\end{theorem}The issue of proving higher regularity of H\"older continuous solutions to a fractional advection-diffusion equation has been considered in the context of the supercritical SQG equation~\cite{ConstWu08}, the critical SQG equation~\cite{CaffVass}, the modified critical SQG equation ~\cite{ConstIyerWu,MiaoXue}, and in the recent preprint~\cite{Silvestre} which addresses a linear equation with singular drift. The natural characterization of H\"older spaces in terms of Besov spaces were utilized in \cite{ConstIyerWu} for \eqref{eq:MQG:1}--\eqref{eq:MQG:2}, respectively in \cite{ConstWu08} for the supercritical SQG equation, to prove that if a solution is $C^\alpha$ for some $\alpha\in(0,1)$, then in fact it lies in a more regular H\"older space, which can be bootstrapped to prove the classical solution is classical.

The techniques used in \cite{ConstIyerWu,ConstWu08} may be applied in order to prove higher regularity for the system \eqref{eq:N1}--\eqref{eq:N4}, but only once the $C^\alpha$ regularity is such that $\alpha>1/2$ (this was also pointed out in \cite{MiaoXue}). However, if we only know that a weak solution of \eqref{eq:N1}--\eqref{eq:N4} is in $C^\alpha$ with $\alpha\in(0,1/2)$, the velocity field is too rough, and the aforementioned method of ~\cite{ConstIyerWu,ConstWu08} does not apply directly. We find that it is necessary to use different arguments to obtain the desired result. By working in the Chemin-Lerner space-time Besov spaces (see \cite{CheminLerner}) we make use of the smoothing effect of the Laplacian at the level of each frequency shell, which enables us to take advantage of the extra a priori information that $u\in L_{t,x}^2$. The principal difficulty lies in treating the high-high frequency interaction in the paraproduct decomposition of the nonlinear term (cf.~\eqref{eq:Bony} below). The main result of this paper is the proof of higher regularity for solutions of \eqref{eq:N1}--\eqref{eq:N4} \textit{without any minimality requirement} on $\alpha>0$, and in any dimension $d\geq 2$ (cf.~Theorem~\ref{thm:higher}). The method introduced in order to prove Theorem~\ref{thm:higher} also gives new higher regularity results for the system \eqref{eq:MQG:1}--\eqref{eq:MQG:2}, in the parameter range $\beta \in (1,2)$ (cf.~Theorem~\ref{thm:MQG} below).

This paper is organized as follows. In Section~\ref{sec:preliminaries} we recall some facts about Besov spaces. Section~\ref{sec:proof:easy} gives the proof of Theorem~\ref{thm:higher} for $\alpha >1/2$, while for the case $0<\alpha<1/2$ the proof is given in Section~\ref{sec:proof:hard}. Section~\ref{sec:MQG} contains a description of our results for the modified critically dissipative SQG equations.

%\begin{remark}
%  The statement of Theorem~\ref{thm:higher} is equivalent to the following. Assume we have the parabolic equation
%  \begin{align}
%    \partial_t \theta - \partial_i (a_{ij} \partial_j \theta) = 0
%  \end{align}where $a_{ij}$ are uniformly elliptic, and $a_{ij} \in C^\alpha$. Then $\theta \in C^{1,\alpha}$.
%\end{remark}

\section{Preliminaries} \label{sec:preliminaries} \setcounter{equation}{0}
 Let $\{\hat{\phi}_j\}_{j\in {\mathbb Z}}$ be a standard dyadic decomposition of the frequency space $\RRd$, with the Fourier support of the Schwartz function $\hat{\phi}_j$ being $\{2^{j-1} \leq |\xi| \leq 2^{j+1}\}$, and where $\sum_j \hat{\phi}_j(\xi) = 1$ on $\RRd\setminus\{0\}$. As usual, define $\Delta_j f = \phi_j \ast f$ and $S_j = \sum_{k<j} \Delta_j f$ for all Schwartz functions $f$.

 For $s\in{\mathbb R}$ and $1\leq p,q \leq \infty$ the homogeneous Besov norm $\dot{B}^{s}_{p,q}$ is classically defined as
 \begin{align}
   \Vert f  \Vert_{\dot{B}_{p,q}^s} = \left\Vert 2^{js} \Vert \Dj f \Vert_{L^p} \right\Vert_{\ell^{q}({\mathbb Z})},
 \end{align}whenever $q\in[1,\infty)$, while in the case $q=\infty$ one defines
 \begin{align}
 \Vert f \Vert_{\BB{s}{p}} = \sup_{j\in {\mathbb Z}} 2^{js} \Vert \Delta_j f \Vert_{L^p}.
 \end{align}Recall that $L^\infty \cap \BB{s}{\infty} = C^s$ is the H\"older space with index $s$, except when $s$ is a nonnegative integer (then we recover the Zygmund spaces $C_*^s$). For any $r\in[1,\infty]$ we  classically let $L^r(\II;\dot{B}^{s}_{p,q})$ denote the set of all Bochner $L^r$-integrable functions on $\II$, with values in $\dot{B}^{s}_{p,q}$, where $\II \subset [0,\infty)$ is some given time interval.

 Lastly, for $s\in{\mathbb R}$, a time interval $\II$, and $1\leq r,p,q \leq \infty$ we recall the Chemin-Lerner space-time Besov spaces $\tilde{L}^r(\II;\dot{B}^{s}_{p,q})$, with norm
  \begin{align}
 \Vert f \Vert_{\tilde{L}^r(\II;\dot{B}^{s}_{p,q})} = \left\Vert 2^{js} \left( \int_{\II} \Vert \Delta_j f(\cdot,t) \Vert_{L^p}^r\; dt\right)^{1/r} \right\Vert_{\ell^q({\mathbb Z})},
 \end{align}with the usual convention of taking a supremum in $j$ if $q=\infty$. Note that $\tilde{L}^r(\II;\dot{B}^{s}_{p,r}) = L^r(\II;\dot{B}^{s}_{p,r})$ for all $r\geq 1$.

\section{Proof of the main theorem in the case $\alpha \in (1/2,1)$} \label{sec:proof:easy} \setcounter{equation}{0}

 In this case the proof follows directly from \cite{ConstIyerWu,ConstWu08}, with only slight modifications, so we give very few details. First, note that if $\theta$ is as in the statement of the lemma, then $\theta \in L^{\infty}([t_0,\infty); \BB{\alpha_p}{p})$, where $\alpha_p = (1-2/p)\alpha$, and $p\in[2,\infty)$ is fixed, to be chosen later. Then, for $j\in {\mathbb Z}$ fixed, we have
  \begin{align}\label{eq:CIW1}
    \frac{1}{p}\frac{d}{dt} \Vert \Delta_j \theta \Vert_{L^p}^p +\int |\Delta_j \theta|^{p-2} \Delta_j \theta (-\Delta) \Delta_j \theta  = - \int |\Delta_j \theta|^{p-2} \Delta_j \theta \Delta_j (u\cdot \nabla \theta).
  \end{align}Upon integration by parts and using \cite[Proposition 29.1]{Lemarie} (cf.~\cite{Planchon}, see also~\cite{ChenMiaoZhang,Wu} for the fractionally diffusive case), the dissipative term is bounded from below as
  \begin{align}\label{eq:CIW2}
     \int |\Delta_j \theta|^{p-2} \Delta_j \theta (-\Delta) \Delta_j \theta\; dx \geq \frac{2^{2j}}{C} \Vert \Delta_j \theta \Vert_{L^p}^p,
  \end{align}where $C=C(d,p)>0$ is a sufficiently large constant. The main difficulty lies in estimating the convection term. This is achieved in \cite{ConstIyerWu,ConstWu08} by using the Bony paraproduct decomposition
  \begin{align}
    \Dj(u \cdot \nabla \theta) = \sum_{|j-k|\leq 2} \Delta_j \nabla \cdot (S_{k-1} u \Delta_k \theta) + \sum_{|j-k|\leq 2} \Delta_j (\Delta_k u \cdot \nabla S_{k-1} \theta) + \sum_{k\geq j-1} \sum_{|k-l|\leq 2} \Delta_j \nabla \cdot (\Delta_k u \Delta_l \theta).\label{eq:Bony1}
  \end{align}When integrated against $\Dj \theta |\Dj \theta|^{p-2}$, \eqref{eq:Bony1} gives rise to three terms on the right side of \eqref{eq:CIW1}. The first two terms (when $|j-k|\leq 2$) may be bounded favorably for any $\alpha >0$, by first integrating by parts the derivative contained in $S_{k-1}u_i =  \partial_j S_{k-1}T_{ij} \theta$, then using a commutator estimate, the H\"older and Bernstein inequalities (see~\cite{ConstIyerWu} for details). However, the third term on the right side of \eqref{eq:Bony1} gives rise to an integral which may only be bounded favorably when $\alpha \in (1/2,1)$. Indeed, from the H\"older inequality we obtain
  \begin{align}
    \sum_{k\geq j-1} \sum_{|k-l|\leq 2} \left| \int \Delta_j \nabla \cdot (\Delta_k u \Delta_l \theta) \Dj \theta |\Dj \theta|^{p-2} \right| &\leq C \Vert \Dj \theta \Vert_{L^p}^{p-2} 2^j \sum_{k\geq j-1} \sum_{|k-l|\leq 2} \Vert \Delta_k u \Vert_{L^p} \Vert \Delta_l \theta \Vert_{L^\infty} \notag\\
    & \leq C \Vert \Dj \theta \Vert_{L^p}^{p-2} 2^j \Vert \theta \Vert_{C^\alpha} \sum_{k\geq j-1} 2^{k(1-\alpha - \alpha_p)} 2^{k \alpha_p} \Vert \Delta_k \theta \Vert_{L^p}.\label{eq:CIW:4}
  \end{align}Given that $\theta \in \dot{B}^{\alpha_p}_{p,\infty}$, the sum of the right side of the above estimate is finite only if $\alpha + \alpha_p > 1$. The later holds if and only if $\alpha >1/2$ (and $p$ is large enough, depending on $\alpha$). However, if $\alpha \in (0,1/2)$ it seems that the method of \cite{ConstIyerWu} cannot be applied directly. We overcome this difficulty in Section~\ref{sec:proof:hard} below. In the case $\alpha\in (1/2,1)$, the right side of \eqref{eq:CIW:4} does remain bounded and the estimate on the nonlinear term may be summarized as
  \begin{align}\label{eq:CIW3}
    \left| \int |\Delta_j \theta|^{p-2} \Delta_j \theta \Delta_j (u\cdot \nabla \theta)\; dx \right| \leq C 2^{(2-2\alpha_p )j} \Vert \theta \Vert_{C^{\alpha_p}} \Vert \theta \Vert_{\BB{\alpha_p}{p}}.
  \end{align}Combining \eqref{eq:CIW1}--\eqref{eq:CIW3} with the Gr\"onwall inequality, and taking the supremum in $j$, we obtain that
  \begin{align}
  \theta \in L^{\infty}([t_1,\infty); \BB{2\alpha_p}{p}(\RRd))
   \end{align} for any $t_1>t_0$. Using the Besov embedding theorem we obtain that $\theta \in L^\infty([t_1,\infty);\BB{2\alpha - \eps_p}{\infty}(\RRd))$, for any $t_1>t_0$, where $\eps_p = (4\alpha+d)/p < (4+d)/p$. Letting $p> (4+d)/(2\alpha-1)$ concludes the proof of the theorem in the case $\alpha\in(1/2,1)$.

\section{Proof of the main theorem in the case $\alpha \in (0,1/2]$} \label{sec:proof:hard}\setcounter{equation}{0}

Let us fix $\II=[t_0,t_1]$, for some $0<t_0<t_1$. The following lemma gives the principal estimate needed in the proof of Theorem~\ref{thm:higher}.
\begin{lemma}\label{lemma:main}
  Let $\theta$ be a weak solution of \eqref{eq:N1}--\eqref{eq:N4} which is H\"older continuous, that is
  \begin{align}
    \theta \in L^\infty(\II;L^2(\RRd)) \cap L^2(\II;\dot{H}^{1}(\RRd)) \cap L^\infty(\II;C^\alpha(\RRd))
  \end{align}for some $\alpha \in (0,1/2)$. If additionally
  \begin{align}
    \theta \in L^2(\II;\dot{B}^{1}_{p,2}(\RRd)),
  \end{align}for some $p \geq 2$, then we have
  \begin{align}
    \theta \in \tilde{L}^{2}(\II;\dot{B}^{1}_{q,r}(\RRd))
  \end{align}for all $1\leq r \leq \infty$, and for all $q\in (p, m_\alpha p)$, where $m_\alpha = (1-\alpha)/(1-2\alpha) > 1$.
\end{lemma}
\begin{proof}
  Apply $\Dj$ to \eqref{eq:N1}, multiply by $\Dj \theta |\Dj \theta|^{q-2}$, integrate over $\RRd$, and use~\cite[Proposition 29.1]{Lemarie} (cf.~\cite{ChenMiaoZhang,Planchon,Wu}), to obtain
 \begin{align}\label{eq:ODE1}
    \frac{1}{q} \frac{d}{dt} \Vert \Dj \theta \Vert_{L^q}^q + c  2^{2j} \Vert \Dj \theta \Vert_{L^q}^q &\leq \left| \int \Dj(u \cdot \nabla \theta) \Dj \theta |\Dj \theta|^{q-2}\right|
 \end{align}for some constant $c=c(d,q)>0$. Using the Bony paraproduct decomposition and the divergence free nature of $u$ (and hence of $S_{k-1} u$ and $\Delta_k u$) we write
 \begin{align}\label{eq:Bony}
  \Dj(u \cdot \nabla \theta) = \sum_{|j-k|\leq 2} \Delta_j \nabla \cdot (S_{k-1} u \Delta_k \theta) + \sum_{|j-k|\leq 2} \Delta_j (\Delta_k u \cdot \nabla S_{k-1} \theta) + \sum_{k\geq j-1} \sum_{|k-l|\leq 2} \Delta_j \nabla \cdot (\Delta_k u \Delta_l \theta).
 \end{align}From  the H\"older inequality, \eqref{eq:ODE1}, and \eqref{eq:Bony} we hence obtain
 \begin{align}
    \frac{1}{q} \frac{d}{dt} \Vert \Dj \theta \Vert_{L^q}^q + c  2^{2j} \Vert \Dj \theta \Vert_{L^q}^q \leq C \Vert \Dj \theta \Vert_{L^q}^{q-1} \left( J_1 + J_2 + J_3\right) \label{eq:lemma:1}
 \end{align}where $c=c(d,p)>0$ is a sufficiently small constant, and we have denoted
 \begin{align}
   J_1 &=  \sum_{|j-k|\leq 2}  \Vert \Dj \nabla \cdot (S_{k-1}u\, \Delta_k \theta)\Vert_{L^q} \label{eq:J1:def}\\
   J_2 &=  \sum_{|j-k|\leq 2}  \Vert \Dj (\Delta_k u \cdot \nabla S_{k-1} \theta)\Vert_{L^q} \label{eq:J2:def} \\
   J_3 &=  \sum_{k \geq j-1} \sum_{|k-l|\leq 2} \Vert \Dj \nabla \cdot (\Delta_k u\, \Delta_l \theta)\Vert_{L^q}\label{eq:J3:def}.
 \end{align}We bound $J_1$ using the Bernstein inequality, the boundedness of Calder\'on-Zygmund operators on $L^q$, and the triangle inequality $\Vert S_{k-1} u \Vert_{L^q} \leq \sum_{l < k-1} \Vert \Delta_l u \Vert_{L^q}$ for all $q \in [1, \infty]$ (note that $\hat{u}(0)=0$), to obtain
 \begin{align}
   J_1 &\leq C 2^j \sum_{|j-k|\leq 2} \sum_{l < k-1} \Vert \Delta_l u \Vert_{L^q} \Vert \Delta_k \theta \Vert_{L^\infty}\notag\\
   &\leq C 2^j \sum_{|j-k|\leq 2} 2^{-k\alpha}\left( 2^{k \alpha}\Vert \Delta_k \theta \Vert_{L^\infty}\right) \sum_{l < k-1} 2^{l} \Vert \Delta_l \theta \Vert_{L^q}\notag\\
   &\leq C \Vert \theta \Vert_{C^\alpha} 2^j \sum_{|j-k|\leq 2} 2^{-k\alpha} \sum_{l< k-1} 2^{l} \left( 2^{l} \Vert \Delta_l \theta \Vert_{L^p}\right)^{p/q} \left( 2^{l\alpha} \Vert \Delta_l \theta \Vert_{L^\infty}\right)^{1-p/q} 2^{-l(p/q + \alpha(1-p/q))}\notag\\
   &\leq C \Vert \theta \Vert_{C^\alpha}^{2-p/q} 2^j \sum_{|j-k|\leq 2} 2^{-k\alpha} \sum_{l< k-1} 2^{l(1-p/q - \alpha(1-p/q))} \left( 2^{l} \Vert \Delta_l \theta \Vert_{L^p}\right)^{p/q} \notag\\
   &\leq C \Vert \theta \Vert_{C^\alpha}^{2-p/q} 2^{j(1-\alpha)} \sum_{l\leq j} 2^{l(1-p/q - \alpha(1-p/q))} \left( 2^{l} \Vert \Delta_l \theta \Vert_{L^p}\right)^{p/q} \label{eq:J1:bound}.
 \end{align}In the above estimate we also used the interpolation inequality $\Vert f \Vert_{L^q} \leq \Vert f \Vert_{L^p}^{p/q} \Vert f \Vert_{L^\infty}^{1-p/q}$, which holds for all functions $f \in L^p \cap L^\infty$, and any $q\in (p,\infty)$. Note that since $\alpha < 1$ we have
 \begin{align}
   1-\frac pq - \alpha \left( 1- \frac pq \right) > 0 \Leftrightarrow 1- \frac pq > 0 \Leftrightarrow q>p,
 \end{align} so that for all $q>p$ and all $s\in[1,\infty)$ we have
 \begin{align}
   \left( \sum_{l\leq j} 2^{s l(1-p/q - \alpha(1-p/q))}\right)^{1/s} \leq C  2^{j(1-p/q - \alpha(1-p/q))}. \label{eq:sum:1}
 \end{align}We bound $J_2$ similarly,
 \begin{align}
   J_2 &\leq C \sum_{|j-k|\leq 2} \sum_{l < k-1} \Vert \Delta_k u \Vert_{L^q} \Vert \nabla \Delta_l \theta \Vert_{L^\infty}\notag\\
   &\leq C \sum_{|j-k|\leq 2} 2^k \Vert \Delta_k \theta \Vert_{L^q} \sum_{l < k-1}2^l \Vert \Delta_l \theta \Vert_{L^\infty}\notag\\
   &\leq C \sum_{|j-k|\leq 2} 2^k \Vert \Delta_k \theta \Vert_{L^p}^{p/q} \Vert \Delta_k \theta \Vert_{L^\infty}^{1-p/q}  \sum_{l < k-1} 2^{l(1-\alpha)} \left( 2^{l\alpha} \Vert \Delta_l \theta \Vert_{L^\infty}\right) \notag\\
   &\leq C \Vert \theta \Vert_{C^\alpha} \sum_{|j-k|\leq 2} 2^k \left( 2^k \Vert \Delta_k \theta \Vert_{L^p}\right)^{p/q} \left( 2^{k \alpha} \Vert \Delta_k \theta \Vert_{L^\infty}\right)^{1-p/q}  2^{-k(p/q + \alpha(1-p/q))} 2^{k(1-\alpha)}\notag\\
   &\leq C \Vert \theta \Vert_{C^\alpha}^{2-p/q} 2^{j(2-\alpha-p/q-\alpha(1-p/q))} \sum_{|j-k|\leq 2} \left( 2^k \Vert \Delta_k \theta \Vert_{L^p}\right)^{p/q}.\label{eq:J2:bound}
 \end{align}Note that here we  used $\alpha <1$ to obtain that $\sum_{l < k-1} 2^{l(1-\alpha)} \leq C 2^{k(1-\alpha)}$. Lastly, we bound $J_3$ as
 \begin{align}
    J_3 &\leq C 2^{j} \sum_{k\geq j-1} \sum_{|k-l|\leq 2} \Vert \Delta_k u \Vert_{L^q} \Vert \Delta_l \theta \Vert_{L^\infty}\notag\\
    &\leq C 2^j \sum_{k\geq j-1} 2^k \Vert \Delta_k \theta \Vert_{L^q} \sum_{|k-l|\leq2} 2^{-l\alpha} \left( 2^{l\alpha} \Vert \Delta_l \theta \Vert_{L^\infty}\right)\notag\\
    &\leq C 2^{j} \Vert \theta \Vert_{C^\alpha} \sum_{k\geq j-1} 2^{k(1-\alpha)} \Vert \Delta_k \theta \Vert_{L^p}^{p/q} \Vert \Delta_k \theta \Vert_{L^\infty}^{1-p/q}\notag\\
    &\leq C 2^{j} \Vert \theta \Vert_{C^\alpha} \sum_{k\geq j-1} 2^{k(1-\alpha)} \left( 2^k \Vert \Delta_k \theta \Vert_{L^p}\right)^{p/q} \left( 2^{k \alpha} \Vert \Delta_k \theta \Vert_{L^\infty}\right)^{1-p/q} 2^{-k(p/q + \alpha(1-p/q))}\notag\\
    &\leq C 2^{j} \Vert \theta \Vert_{C^\alpha}^{2-p/q} \sum_{k\geq j-1} 2^{k (1 - \alpha - p/q - \alpha(1-p/q))} \left( 2^k \Vert \Delta_k \theta \Vert_{L^p}\right)^{p/q}.\label{eq:J3:bound}
  \end{align}Here as before we used the Bernstein's inequality and the bound $\Vert \Delta_k u \Vert_{L^q} \leq C 2^k \Vert \Delta_k \theta \Vert_{L^q}$. If we let
  \begin{align}
   p <  q < m_\alpha p = \frac{1-\alpha}{1-2\alpha} p,
  \end{align}the exponent of $2^k$ in the last inequality of \eqref{eq:J3:bound} lies in the range
  \begin{align}
  -\alpha < 1- \alpha - \frac pq -\alpha \left(1- \frac pq\right)<  1 - \frac{1}{m_\alpha} - \alpha \left( 2 - \frac{1}{m_\alpha}\right) = 0
  \end{align}since $\alpha \in(0,1/2)$. Therefore, if $q\in(p,m_\alpha p)$, for any $s\in[1,\infty)$  we have
   \begin{align}
     \left( \sum_{k\geq j-1}  2^{s k (1 - \alpha - p/q - \alpha(1-p/q))} \right)^{1/s} \leq C 2^{j (1 - \alpha - p/q - \alpha(1-p/q))}.\label{eq:sum:2}
   \end{align}We insert the bounds \eqref{eq:J1:bound}, \eqref{eq:J2:bound}, and \eqref{eq:J3:bound} into \eqref{eq:lemma:1} and obtain the a priori estimate
  \begin{align}
    \frac{d}{dt} \Vert \Dj \theta \Vert_{L^q} + c  2^{2j} \Vert \Dj \theta \Vert_{L^q} & \leq C \Vert \theta \Vert_{C^\alpha}^{2-p/q} 2^{j(1-\alpha)} \sum_{k\leq j} 2^{k(1-p/q - \alpha(1-p/q))} \left( 2^{k} \Vert \Delta_k \theta \Vert_{L^p}\right)^{p/q} \notag\\
    &\qquad + C \Vert \theta \Vert_{C^\alpha}^{2-p/q} 2^{j(2-\alpha-p/q-\alpha(1-p/q))} \sum_{|j-k|\leq 2} \left( 2^k \Vert \Delta_k \theta \Vert_{L^p}\right)^{p/q}\notag\\
    &\qquad + C \Vert \theta \Vert_{C^\alpha}^{2-p/q} 2^{j} \sum_{k\geq j-1} 2^{k (1 - \alpha - p/q - \alpha(1-p/q))} \left( 2^k \Vert \Delta_k \theta \Vert_{L^p}\right)^{p/q}\label{eq:lemma:2}.
  \end{align}We apply Gr\"onwall's inequality and obtain
  \begin{align}
    \Vert \Dj \theta(t) \Vert_{L^q} &\leq e^{-c 2^{2j} (t-t_0)} \Vert \Dj \theta(t_0) \Vert_{L^q}\notag\\
   &\qquad + C \Vert \theta \Vert_{L^\infty(\II; C^\alpha)}^{2-p/q} 2^{j(1-\alpha)}  \sum_{k \leq j} 2^{k(1 - p/q - \alpha(1-p/q))} \int_{t_0}^{t} e^{-c (t-s) 2^{2j}}  \left( 2^{k} \Vert \Delta_k \theta(s) \Vert_{L^p}\right)^{p/q}  ds\notag\\
   &\qquad + C \Vert \theta \Vert_{L^\infty(\II; C^\alpha)}^{2-p/q} 2^{j(2-\alpha-p/q-\alpha(1-p/q))}   \sum_{|k -j|\leq 2}  \int_{t_0}^{t} e^{-c (t-s) 2^{2j}}  \left( 2^{k} \Vert \Delta_k \theta(s) \Vert_{L^p}\right)^{p/q}  ds\notag\\
   &\qquad + C \Vert \theta \Vert_{L^\infty(\II; C^\alpha)}^{2-p/q} 2^{j}  \sum_{k \geq j-1} 2^{k(1 - \alpha- p/q - \alpha(1-p/q))} \int_{t_0}^{t} e^{-c (t-s) 2^{2j}}  \left( 2^{k} \Vert \Delta_k \theta(s) \Vert_{L^p}\right)^{p/q}  ds.\label{eq:lemma:3}
  \end{align}Using the Young-type inequality
  \begin{align}
  \Vert f\ast g \Vert_{L^2(\II)}  \leq  \Vert f \Vert_{L^{1}(\II)} \Vert g \Vert_{L^{2}(\II)}  \leq  \Vert f \Vert_{L^{1}(\II)} \Vert g \Vert_{L^{2q/p}(\II)} |\II|^{(q-p)/2q}\label{eq:lemma:Young}
  \end{align}and the bound
  \begin{align}
  \left\Vert e^{-c 2^{2j} (t-t_0)} \right\Vert_{L^r(\II)} \leq \min\{ C 2^{-2j/r}, |\II|^{1/r} \} \label{eq:kernel:Lr}
\end{align}with $r=1$, we obtain that
\begin{align}
  \left \Vert \int_{t_0}^{t} e^{-c (t-s) 2^{2j}}  \left( 2^{k} \Vert \Delta_k \theta(s) \Vert_{L^p}\right)^{p/q}  ds \right\Vert_{L^2(\II)} \leq  C \min\{2^{-2j}, |\II|\} \left( 2^k \Vert \Delta_k \theta \Vert_{L^2(\II;L^p)}\right)^{p/q} |\II|^{(q-p)/2q}. \label{eq:convolution:L2}
\end{align}We take the $L^2(\II)$ norm of \eqref{eq:lemma:3}, use the bound \eqref{eq:convolution:L2} above, combined with the discrete H\"older inequality, the fact that $\theta \in L^2(\II;\dot{B}^{1}_{p,2})=\tilde{L}^2(\II;\dot{B}^{1}_{p,2})$, the estimates \eqref{eq:sum:1} and \eqref{eq:sum:2} with $s=2q/(2q-p)$, and obtain that
\begin{align}
\Vert \Dj \theta(t) \Vert_{L^2(\II;L^q)} &\leq C \Vert \Dj \theta(t_0) \Vert_{L^p}^{p/q} \Vert \Dj \theta(t_0) \Vert_{L^\infty}^{1-p/q} \min\{2^{-j},|\II|^{1/2}\}\notag\\
& \ \ \ \ + C \Vert \theta \Vert_{L^\infty(\II; C^\alpha)}^{2-p/q} \Vert \theta \Vert_{L^2(\II;\dot{B}^{1}_{p,2})}^{p/q} |\II|^{(q-p)/2q}\left( 2^{j(2-\alpha-p/q-\alpha(1-p/q))}  \min\{ C 2^{-2j}, |\II|\}\right)\notag\\
&\leq C \Vert \theta(t_0) \Vert_{L^p}^{p/q} \Vert \theta(t_0) \Vert_{C^\alpha}^{1-p/q} \left( 2^{-j\alpha (1-p/q)} \min\{2^{-j},|\II|^{1/2}\}\right)\notag\\
& \ \ \ \ + C \Vert \theta \Vert_{L^\infty(\II; C^\alpha)}^{2-p/q} \Vert \theta \Vert_{L^2(\II;\dot{B}^{1}_{p,2})}^{p/q} |\II|^{(q-p)/2q} \left(2^{j(2-\alpha-p/q-\alpha(1-p/q))}  \min\{ C 2^{-2j}, |\II|\}\right)\label{eq:lemma:4}
\end{align}
for all $q\in(p,m_\alpha p)$. Note that $\theta(t_0) \in L^p$ since we a priori have $\theta \in L_t^\infty L_x^2 \cap
L_t^\infty L_x^\infty$. We multiply the above estimate on both sides by $2^{j}$ and take an $\ell^r({\mathbb Z})$-norm, to obtain that
  \begin{align}
    \Vert \theta \Vert_{\tilde{L}^2(\II;\dot{B}^{1}_{q,r})} &= \left\Vert 2^{j} \Vert \Dj \theta \Vert_{L^2(\II;L^q)} \right\Vert_{\ell^r({\mathbb Z})} \notag\\
    & \leq \Vert \theta(t_0) \Vert_{L^p}^{p/q} \Vert \theta(t_0) \Vert_{C^\alpha}^{1-p/q} \left\Vert 2^{j(1-\alpha (1-p/q))} \min\{2^{-j},|\II|^{1/2}\} \right\Vert_{\ell^r({\mathbb Z})}\notag\\
    &\ \ \ \ +C \Vert \theta \Vert_{L^\infty(\II; C^\alpha)}^{2-p/q} \Vert \theta \Vert_{L^{2}(\II;\dot{B}^{1}_{p,2})}^{p/q} |\II|^{(q-p)/2q} \left\Vert 2^{j(3-\alpha-p/q-\alpha(1-p/q))}  \min\{ C 2^{-2j}, |\II|\} \right\Vert_{\ell^r({\mathbb Z})}.\label{eq:lemma:5}
  \end{align}The key observation is that for all $q\in(p,m_\alpha p)$ we have  $- \alpha(1-p/q) < 0$, $1-\alpha(1-p/q)>0$, $1 - \alpha - p/q - \alpha(1-p/q))<0$, and $3 - \alpha - p/q - \alpha(1-p/q))>0$, so that the two $\ell^r$ norms on the right side of the above estimate are finite, for any $1\leq r \leq \infty$. We have hence proven that
  \begin{align}
    \theta \in \tilde{L}^{2}(\II;\dot{B}^1_{q,r})
  \end{align}for any $1\leq r \leq \infty$, and any $q \in (p, m_\alpha p)$, concluding the proof of the lemma.
\end{proof}
The following lemma shows how one may bootstrap the arguments in Lemma~\ref{lemma:main} in order to control $\theta$ in $L_t^2 W_x^{1,\infty}$.
\begin{lemma}\label{lemma:bootstrap}
  Let $\theta$ be a weak solution of \eqref{eq:N1}--\eqref{eq:N4} which is H\"older continuous, that is
  \begin{align}
    \theta \in L^\infty(\II;L^2(\RRd)) \cap L^2(\II;\dot{H}^{1}(\RRd)) \cap L^\infty(\II;C^\alpha(\RRd))\label{eq:bootstrap:cond}
  \end{align}for some $\alpha \in (0,1/2)$. Then we have
  \begin{align}
    \nabla \theta \in L^2(\II;L^{\infty}(\RRd)).
  \end{align}
\end{lemma}

\begin{proof}
We note that $\dot{H}^{1} = \dot{B}^{1}_{2,2}$ so that we may apply Lemma~\ref{lemma:main} with $p=2$, and obtain that $\theta \in L^2(\II;\dot{B}^{1}_{q,2})$ for any $q\in(2,2m_\alpha)$. Since $m_\alpha>1$, we may bootstrap and apply Lemma~\ref{lemma:main} once more to obtain that $\theta \in L^2(\II;\dot{B}^{1}_{q,2})$ for all $q\in(2,2m_\alpha^2)$. For any fixed $p>2$, since $m_\alpha^k$ diverges as $k\rightarrow \infty$, we may iterate Lemma~\ref{lemma:main} finitely many times and obtain that $\theta \in \tilde{L}^{2}(\II;\dot{B}^1_{p,r})$, for all $r\in [1,\infty]$.

Fix $p$ large enough, to be explicitly chosen later, and let $q = p (1+m_\alpha)/2$. From the estimate~\eqref{eq:lemma:4}, for any $\epsilon>0$ we have
\begin{align}
 & 2^{j(1+\epsilon)} \Vert \Dj \theta \Vert_{L^2(\II;L^q)}\notag\\
 & \qquad \leq C \Vert \theta(t_0) \Vert_{L^p}^{p/q} \Vert\theta(t_0) \Vert_{C^\alpha}^{1-p/q} \min\{C 2^{j(\epsilon - \alpha(1-p/q))},|\II|^{1/2} 2^{j(1+\epsilon - \alpha(1-p/q))}\} \notag\\
 &\qquad + C \Vert \theta \Vert_{L^\infty(\II; C^\alpha)}^{2-p/q} \Vert \theta \Vert_{L^{2}(\II;\dot{B}^{1}_{p,2})}^{p/q} |\II|^{(q-p)/2q} \min\{C 2^{j(\epsilon + 1 - p/q - \alpha(2-p/q))},2^{j(\epsilon + 3 - p/q - \alpha(2-p/q))} |\II|\}\label{eq:mainestimate3}
\end{align}where $q = p (1+m_\alpha)/2$. We now pick a suitable $\epsilon>0$, so that the $\ell^1({\mathbb Z})$ norm of the right side of \eqref{eq:mainestimate3} is finite. For this to hold, we need that the following four bounds to hold true
\begin{align}
  \epsilon - \alpha\left(1- \frac{2}{1+m_\alpha} \right) &< 0 \Leftrightarrow \epsilon < \frac{\alpha^2}{2-3\alpha}\label{eq:cond1}\\
  1+\epsilon - \alpha\left(1- \frac{2}{1+m_\alpha} \right)&>0 \Leftrightarrow \epsilon > -\frac{2- 3 \alpha - \alpha^2}{2-3\alpha}\label{eq:cond2}\\
  \epsilon + 1 - \frac{2}{m_\alpha}  - \alpha \left(2- \frac{2}{1+m_\alpha} \right) &<0 \Leftrightarrow \epsilon< \frac{(1-2\alpha)(2-3\alpha -\alpha^2)}{(1-\alpha)(2-3\alpha)}\label{eq:cond3}\\
  \epsilon + 3 - \frac{2}{m_\alpha}  - \alpha\left(2- \frac{2}{1+m_\alpha} \right) &>0 \Leftrightarrow \epsilon>- \frac{2 - 3 \alpha + \alpha^2 -2\alpha^3}{(1-\alpha)(2-3\alpha)}\label{eq:cond4}
\end{align}where we used that $p/q = 2/(1+m_\alpha)$. Note that $2-3\alpha -\alpha^2>0$ and $2 - 3 \alpha + \alpha^2 -2\alpha^3>0$ whenever $0<\alpha < 1/2$, so that we must choose $\epsilon$ such that only \eqref{eq:cond1} and \eqref{eq:cond3} hold. It is therefore sufficient to let
\begin{align}
  \epsilon_\alpha = \frac{1}{2} \min\left\{ \frac{\alpha^2}{2-3\alpha}, \frac{(1-2\alpha)(2-3\alpha -\alpha^2)}{(1-\alpha)(2-3\alpha)}\right\}.
\end{align}It can be easily checked that for any $\alpha\in(0,1/2)$ we have $\epsilon_\alpha>0$. With this choice of $\epsilon=\epsilon_\alpha$ we may take the $\ell^1$ norm of \eqref{eq:mainestimate3} and obtain that
\begin{align}
  \theta \in \tilde{L}^2(\II;\dot{B}^{1+\epsilon_\alpha}_{p(1+m_\alpha)/2,1}) \subset L^2(\II;\dot{B}^{1+\epsilon_\alpha}_{p(1+m_\alpha)/2,1}).
\end{align}The Besov embedding theorem $\dot{B}^{s}_{p,1} \subset \dot{B}^{s-d/p}_{\infty,1}$ gives that
\begin{align}
  \dot{B}^{1+\epsilon_\alpha}_{p(1+m_\alpha)/2,1} \subset \dot{B}^{1+\epsilon_\alpha - 2d/(p + p m_\alpha)}_{\infty,1},\label{eq:u:Besov}
\end{align}so that choosing $p=p(\alpha,d)>2$ to satisfy
\begin{align}
  \epsilon_\alpha - \frac{2d}{p(1 +  m_\alpha)} = 0 \label{eq:pchoice}
\end{align}we obtain
\begin{align}
 \nabla \theta \in L^2(\II;\dot{B}^{0}_{\infty,1}).\label{eq:dotbesov}
\end{align}We note that we may explicitly solve for $p$ in \eqref{eq:pchoice}
\begin{align}
  p = \frac{2d}{\epsilon_\alpha(1+m_\alpha)} \geq 4d > 2
\end{align}for any $\alpha \in (0,1/2)$.  We recall from \eqref{eq:bootstrap:cond} that $\nabla \theta \in L^2(\II;L^2)$, and hence  $\nabla \theta \in L^2(\II; L^2 \cap \dot{B}^{0}_{\infty,1})$, by \eqref{eq:dotbesov}. Lastly, we use $L^2 \cap \dot{B}^{0}_{\infty,1} \subset B^{0}_{\infty,1}$ and the borderline Sobolev embedding theorem
\begin{align}
{B}^{0}_{\infty,1} \subset L^\infty
\end{align}to obtain that
\begin{align}
 \nabla \theta \in L^2(\II;L^{\infty})
\end{align}which concludes the proof of the lemma. Note that by \eqref{eq:u:Besov} we may even obtain that $u\in L^2(\II; L^\infty)$.
\end{proof}

A simple consequence of this improved regularity is the following statement.
\begin{proposition}\label{prop:main}
  Let $\theta$ be a H\"older continuous weak solution of \eqref{eq:N1}--\eqref{eq:N4} such that
  \begin{align}
    \nabla \theta \in L^2([t_0,t_1];L^\infty).
  \end{align}Then
  \begin{align}
    \theta \in L^\infty([t_2,t_1];\dot{H}^m)
  \end{align}for any $m\geq 2$, and a.e. $t_2 \in (t_0,t_1)$.
\end{proposition}

\begin{proof}[Proof of Proposition~\ref{prop:main}]Since $u$ is divergence free, we have the a priori estimate
  \begin{align}
    \frac{1}{2} \frac{d}{dt} \Vert \nabla \theta \Vert_{L^2}^2 + \Vert \Delta \theta \Vert_{L^2}^2 \leq \left| \int \partial_k u_j \partial_k \theta \partial_j \theta \right| \leq \Vert \Delta \theta \Vert_{L^2} \Vert \nabla \theta \Vert_{L^2} \Vert \nabla \theta \Vert_{L^\infty} \leq \frac{1}{2}  \Vert \Delta \theta \Vert_{L^2}^2 + \frac 12 \Vert \nabla \theta \Vert_{L^2}^2 \Vert \nabla \theta \Vert_{L^\infty}^2.
  \end{align}We absorb the $(1/2) \Vert \Delta\theta \Vert_{L^2}^2$ term on the left side of the above estimate and obtain that
  \begin{align}
    \Vert \theta(t) \Vert_{\dot{H}^1}^2 \leq \Vert \theta(t_0) \Vert_{\dot{H}^1}^2 e^{ \int_{t_0}^{t} \Vert \nabla \theta(s) \Vert_{L^\infty}^2\; ds}
  \end{align}which is finite for all $t_0 \leq t\in \II$ thanks to the assumption $\nabla \theta \in L^2(\II;L^\infty)$, as long as $\theta(t_0) \in \dot{H}^1$. The latter is true for a.e. $t_0>0$ since we a priori knew that $\theta \in L^2((0,\infty);\dot{H}^{1})$, and $L^2$ functions are finite a.e. (by using arguments similar to \cite[Chapter 9]{ConstantinFoias}, one may even obtain explicit bounds in terms of $\Vert \theta_0\Vert_{L^2}$). This shows that $\theta\in L^\infty(\II;\dot{H}^1) \cap L^2(\II;\dot{H}^2)$. Repeating the above argument with $\II=[t_0,t_1]$ replaced by some $\II'=[t_2,t_1]$, where $t_2>t_0$, we get
    \begin{align}
    \frac{1}{2} \frac{d}{dt} \Vert \Delta \theta \Vert_{L^2}^2 + \Vert \nabla \Delta \theta \Vert_{L^2}^2 \leq \left| \int \Delta (u_j \partial_j \theta) \Delta \theta \right| \leq C \Vert \nabla \Delta \theta \Vert_{L^2} \Vert \Delta \theta \Vert_{L^2} \Vert \nabla \theta \Vert_{L^\infty},
  \end{align}on $\II'$, and therefore obtain $\theta\in L^\infty(\II';\dot{H}^2) \cap L^2(\II';\dot{H}^3)$. Hence, for any $m\geq 2$, finitely many iterations of the above argument proves that $\theta\in L^\infty(\II'';\dot{H}^m) \cap L^2(\II'';\dot{H}^{m+1})$ for any $\II''\subset \II$, concluding the proof.\end{proof}

\begin{proof}[Proof of Theorem~\ref{thm:higher}]
If $\alpha \in (1/2,1)$ the theorem follows from the arguments given in Section~\ref{sec:proof:easy}. If $\alpha =1/2$, we simply consider the solution to lie in $C^{1/2-\epsilon} \subset L^\infty \cap C^{1/2}$, for some $\epsilon >0$, reducing the proof of the theorem to the case $\alpha \in (0,1/2)$. In this case, we apply Lemma~\ref{lemma:bootstrap} to obtain that $\nabla \theta \in L^2(\II;L^\infty)$ for any $\II=[t_0,t_1] \subset (0,\infty)$. Using Proposition~\ref{prop:main}, this implies that $\theta \in L^\infty([t_2,t_1];H^m)$ for  some large enough $m$ (any $m>d/2+1$ is sufficient), and a.e. $t_2 \in (t_0,t_1)$. The statement of the theorem follows from the Sobolev the embedding $H^m(\RRd) \subset C^{1,\delta}(\RRd)$, for some $\delta\in(0,1)$. In particular, if $d=3$ one may let $m=3$.
\end{proof}

\begin{corollary}
  The H\"older continuous weak solution $\theta$ of \eqref{eq:N1}--\eqref{eq:N4} is $C^\infty$ smooth for positive time.
\end{corollary}
\begin{proof}
  Let $\tilde{\theta} = \partial_i \theta$ for some $i\in \{1,\ldots,d\}$. From Theorem~\ref{thm:higher} we obtain that for any $0<t_1<T$ we have $\tilde{\theta} \in L^\infty([t_1,T];C^{\delta}(\RRd))$ for some $\delta >0$. From Proposition~\ref{prop:main} we also have that $\tilde{\theta} \in L^\infty([t_1,T];L^2(\RRd)) \cap L^2([t_1,T];\dot{H}^1(\RRd))$. Lastly, the equation satisfied by $\tilde{\theta}$, obtained by applying $\partial_{i}$ to \eqref{eq:N1}, is
  \begin{align}
  \partial_t \tilde{\theta} - \Delta \tilde{\theta} + \partial_j \theta  T_{ij} \partial_i \tilde{\theta} + \partial_i T_{ij} \theta \partial_j \tilde{\theta} = 0 \label{eq:tildetheta}
  \end{align}where we use the summation convention over repeated indexes, and $T_{ij}$ are Calder\'on-Zygmund operators. Given that the coefficients $\partial_j \theta, \partial_i T_{ij} \theta \in L^\infty([t_1,T];C^\delta(\RRd))$ (cf.~Theorem~\ref{thm:higher}) are smoother than the a priori smoothness of the velocity in \eqref{eq:N1} (which belonged to a H\"older space of negative index), it is straightforward to repeat the arguments used to prove Theorem~\ref{thm:higher} in order to show that $\tilde{\theta} \in L^\infty([t_2,T];C^{1,\gamma}(\RRd))$ for some $\gamma >0$, and a.e. $t_2 \in (t_1,T)$. Since $i\in  \{1,\ldots,d\}$ was arbitrary, this shows the solution $\theta$ is $C^{2,\gamma}$ for some $\gamma>0$. The proof of the corollary is concluded by further taking derivatives of the equation, and iterating the above arguments.
\end{proof}

\section{Higher regularity for the critically dissipative modified SQG equations}\label{sec:MQG} \setcounter{equation}{0}

Here we address the applicability of the method presented in Section~\ref{sec:proof:hard} above, to prove higher regularity for the modified critically dissipative SQG equation \eqref{eq:MQG:1}--\eqref{eq:MQG:2}, for the parameter range  $\beta \in (1,2)$. Note that when $\beta = 2$ the equations \eqref{eq:MQG:1}--\eqref{eq:MQG:2} reduce to the heat equation, and regularity is trivial. In \cite{MiaoXue}, Miao and Xue prove the global existence of weak solutions $\theta \in L^\infty([0,\infty);L^2) \cap L^2((0,\infty); \dot{H}^{\beta/2})$, using methods similar to \cite{ConstantinFoias,FriedlanderVicol,Temam}, the local existence of smooth solutions, and the eventual regularity of the weak solutions (see also \cite{ChaeConstWu1,ChaeConstWu2,CCCGW,Kiselev} and references therein, for further results concerning  generalizations of the SQG equations). Moreover, in \cite[Proposition 5.1]{MiaoXue}, the authors prove the following regularity criterion: if a weak solution $\theta$ lies in $L^\infty([t_0,\infty);C^\alpha)$, with $\alpha > (\beta-1)/2$, then $\theta \in C^\infty( (t_1,\infty) \times {\mathbb R}^2)$ for any  $t_1 > t_0$. Such a minimality requirement on $\alpha$ seems to be purely technical, as the problem is \textit{subcritical} in $C^\alpha$ for any $\alpha >0$.

The proof of Theorem~\ref{thm:higher} of the present paper directly applies to \eqref{eq:MQG:1}--\eqref{eq:MQG:2}, with $\beta \in (1,2)$, and gives the following regularity criterion for weak solutions.
\begin{theorem}\label{thm:MQG}
Let $\theta_0 \in L^2$ be given, and let $\II=[t_0,t_1] \subset (0,\infty)$. Given a weak solution
\begin{align}
\theta \in L^\infty(\II;L^2({\mathbb R}^2)) \cap L^2(\II;\dot{H}^{\beta/2}({\mathbb R}^2)) \label{eq:thm:MQG:1}
\end{align}of the initial value problem associated to \eqref{eq:MQG:1}--\eqref{eq:MQG:2}, if
\begin{align}
\theta \in L^\infty(\II;C^\alpha({\mathbb R}^2)) \label{eq:thm:MQG:2}
\end{align}where
\begin{align}
  \min\left\{ \frac{2-\beta}{2}, \frac{\beta-1}{2} \right\} < \alpha < 1 \label{eq:thm:MQG:3}
\end{align} and $\beta \in (1,2)$, then
\begin{align}
  \theta \in L^\infty( [t_2,t_1]; C^{1,\delta}({\mathbb R}^2))
\end{align}for some $\delta>0$, and for a.e. $t_2 \in (t_0,t_1)$. Additionally, we have $\theta \in C^\infty( (t_0,t_1]\times {\mathbb R}^2)$.
\end{theorem}
\begin{proof}
  First we note that if $\alpha \in ( (\beta -1)/2, 1)$, for any $\beta \in (1,2)$ this result was proven in~\cite{MiaoXue}. Therefore, in order to complete the proof of the theorem it is left to treat the range $(2-\beta)/2 < (\beta - 1)/2$, which is equivalent to $\beta \in (3/2,2)$, under the regularity criterion that $(2-\beta)/2<\alpha<(\beta-1)/2$. In order to avoid redundancy we only outline the differences with the proof of Theorem~\ref{thm:higher}.

 Using methods directly corresponding to those described in the proof of Lemma~\ref{lemma:main} we first prove that if a weak solution $\theta$ satisfying \eqref{eq:thm:MQG:1}--\eqref{eq:thm:MQG:2} is such that $\theta \in L^2(\II; \dot{B}^{\beta/2}_{p,2})$, for some $p\geq 2$, then $\theta \in \tilde{L}^2(\II;\dot{B}^{\beta/2}_{q,r})$, for any $1\leq r \leq \infty$, and for any value of $q>p$ such that
 \begin{align}
   \frac{q}{p} \in \left( \frac{\beta/2 - \alpha}{\beta - 1 - \alpha}, \frac{\beta/2 - \alpha}{\beta/2 - 2\alpha}\right).\label{eq:MQG:range}
 \end{align}Due to our choice of $\alpha$, the range \eqref{eq:MQG:range} is not empty. Since initially $\theta \in L^2(\II; \dot{B}^{\beta/2}_{2,2})$, we can therefore bootstrap this argument finitely many times, and similarly to the proof of Lemma~\ref{lemma:bootstrap}, we show that for $p$ large enough we have $\theta \in L^2(\II; \dot{B}^{\beta/2+2/p}_{p,1}) \subset L^2(\II; \dot{B}^{\beta/2}_{\infty,1}) $. As in Proposition~\ref{prop:main}, this implies, via energy estimates and interpolation inequalities, that $\theta\in L^\infty([t_2,t_1];\dot{H}^m) \cap L^2([t_2,t_1];\dot{H}^{m+\beta/2})$ for any $m\geq 1$, and a.e. $t_2 \in (t_0,t_1)$. The proof of \eqref{eq:thm:MQG:3} follows now from the Sobolev embedding, while the proof of higher regularity consists of taking derivatives of the equation and repeating the arguments listed above.
\end{proof}

Due to the sub-criticality of the $C^\alpha$ norms, for any $\alpha>0$, with respect to the natural scaling of the equations \eqref{eq:MQG:1}--\eqref{eq:MQG:2}, we conjecture that condition \eqref{eq:thm:MQG:2} may be replaced with $0 <\alpha < 1$, for any $\beta \in (1,2)$.

\subsection*{Acknowledgment} The work of S.F. was in part supported by the NSF grant DMS 0803268. The authors would like
to thank Peter Constantin for fruitful discussions on the topic, and for pointing out the advantage of space-time Besov
spaces.

\end{document}